\pdfoutput=1
\documentclass[10pt]{article} 
\usepackage[utf8]{inputenc} 

\usepackage{geometry} 
\usepackage{amsmath,amssymb,amsfonts}
\usepackage{amsthm,pstricks,enumitem, textcomp, xspace}
\usepackage{todonotes}

\usepackage{fancyhdr} 
\usepackage[hidelinks]{hyperref}

\newtheorem{theorem}{Theorem}[section]
\newtheorem{innertheorembis}{Theorem}
\newenvironment{theorembis}[1]
  {\renewcommand\theinnertheorembis{#1}\innertheorembis}
  {\endinnertheorembis}
\newtheorem{lemma}[theorem]{Lemma}
\newtheorem{definition}[theorem]{Definition}
\let\olddefinition\definition
\renewcommand{\definition}{\olddefinition\normalfont}
\newtheorem{corollary}[theorem]{Corollary}
\newtheorem{proposition}[theorem]{Proposition}

\newcommand{\CCC}{CAT(0) cube complex\xspace}
\newcommand{\R}{\mathbb R }
\newcommand{\Z}{\mathbb Z }

\newcommand{\inter}[2]{[#1 #2]}
\newcommand{\adjP}{adj-P\xspace}
\newcommand{\blockingpair}{p.o.p\xspace}
\newcommand{\blockingpairs}{p.o.p\xspace}

\newcommand{\actson}{\curvearrowright} 
\newcommand{\Stab}{\mathrm{Stab}} 
\newcommand{\homeo}{\cong} 
\newcommand{\uc}[1]{\tilde{#1}} 

\newcommand{\gp}[1]{#1} 
\newcommand{\gpelt}[1]{#1} 

\newcommand{\mfld}[1]{#1} 
\newcommand{\smfld}[1]{#1} 

\newcommand{\CC}[1]{\mathbf{#1}} 
\newcommand{\CCv}[1]{\mathbf{#1}} 
\newcommand{\CCc}[1]{\mathbf{#1}} 
\newcommand{\itvl}[1]{#1} 

\newcommand{\cCC}[1]{{\mathbf{#1}^c}} 
\newcommand{\cCCv}[1]{\mathbf{#1}^c} 
\newcommand{\cCCc}[1]{\mathbf{#1}^c} 

\newcommand{\fCC}[1]{{\mathbf{#1}^f}} 
\newcommand{\fCCv}[1]{\mathbf{#1}^f} 
\newcommand{\fCCc}[1]{\mathbf{#1}^f} 

\newcommand{\oCC}[1]{{\mathbf{#1}^o}} 
\newcommand{\oCCv}[1]{\mathbf{#1}^o} 

\newcommand{\Hyp}[1]{\hat{\mathcal{#1}}} 
\newcommand{\hyp}[1]{\hat{\mathfrak{#1}}} 
\newcommand{\Hypmap}[1]{\hat{#1}_* } 

\newcommand{\Hs}[1]{\mathcal{#1}} 
\newcommand{\hs}[1]{\mathfrak{#1}} 
\newcommand{\comp}[1]{{#1}^*} 
\newcommand{\Hsmap}[1]{#1_*} 

\newcommand{\cHyp}[1]{\hat{\mathcal{#1}}^c} 
\newcommand{\chyp}[1]{\hat{\mathfrak{#1}}^c} 

\newcommand{\cHs}[1]{\mathcal{#1}^c} 
\newcommand{\chs}[1]{\mathfrak{#1}^c} 

\newcommand{\fHyp}[1]{\hat{\mathcal{#1}}^f} 
\newcommand{\fhyp}[1]{\hat{\mathfrak{#1}}^f} 

\newcommand{\fHs}[1]{\mathcal{#1}^f} 
\newcommand{\fhs}[1]{\mathfrak{#1}^f} 

\newcommand{\oHyp}[1]{\hat{\mathcal{#1}}^o} 
\newcommand{\ohyp}[1]{\hat{\mathfrak{#1}}^o} 

\newcommand{\oHs}[1]{\mathcal{#1}^o} 
\newcommand{\ohs}[1]{\mathfrak{#1}^o} 

\newcommand{\simp}[1]{#1} 
\newcommand{\simpv}[1]{#1} 
\newcommand{\simpe}[1]{#1} 
\newcommand{\simpf}[1]{#1} 
\newcommand{\simps}[1]{#1} 

\newcommand{\usimp}[1]{\tilde{#1}} 
\newcommand{\usimpv}[1]{\tilde{#1}} 
\newcommand{\usimpe}[1]{\tilde{#1}} 
\newcommand{\usimpf}[1]{\tilde{#1}} 

\newcommand{\trk}[1]{#1} 
\newcommand{\ptrn}[1]{{\mathcal{#1}}} 
\newcommand{\utrk}[1]{{\tilde{#1}}} 
\newcommand{\uptrn}[1]{{\tilde{\mathcal{#1}}}} 

\title{Resolutions of CAT(0) cube complexes and accessibility properties}
\author{Benjamin Beeker\thanks{Supported in part by a Technion fellowship} ~and Nir Lazarovich\thanks{Supported by the Adams Fellowship Program of the Israel Academy of Sciences and Humanities}}
\date{} 

\begin{document}
\maketitle

\begin{abstract}
In \cite{Dun85}, Dunwoody defined resolutions for finitely presented group actions on simplicial trees, that is, an action of the group on a tree with smaller edge and vertex stabilizers. He, moreover, proved that the size of the resolution is bounded by a constant depending only on the group. Extending Dunwoody's definition of patterns we construct resolutions for group actions on a general finite dimensional \CCC. In dimension two, we bound the number of hyperplanes of this resolution. We apply this result for surfaces and 3-manifolds to bound collections of codimension-1 submanifolds.
\end{abstract}


\section{Introduction}

An important aspect of group actions on trees is Dunwoody's theory of accessibility (see \cite{Dun85}), and, in particular, finding bounds for ``reasonable'' actions on trees. The earliest result in this direction is Grushko's decomposition theorem (see \cite{Gru40}), that implies in particular that there is a bound, depending only on the rank of the group $\gp{G}$, on the number of edge-orbits in a $\gp{G}$-tree with trivial edge stabilizers. An analogous result for 3-manifolds, known as the Kneser prime decomposition theorem (see \cite{Kne29} and \cite{Mil62}) implies that there is a bound, depending only on the compact 3-manifold $\mfld{M^3}$, on the number of embedded, essential, disjoint, non-homotopic spheres in $\mfld{M}$.
In fact, Haken proved that there is a bound on any collection of such $2$-sided subsurfaces (not necessarily spheres) assuming further that the subsurfaces are incompressible and that the manifold is incompressible (see \cite{Hak61}). 

Grushko's result could be seen as a first result towards Dunwoody's accessibility theorem. As part of the proof, Dunwoody introduced two key tools: patterns and resolutions. 
He observed that any action of an almost finitely presented group, $\gp{G}$, on a tree could be resolved to a $\gp{G}$-tree obtained from a geometric pattern on the universal cover of the presentation complex of $\gp{G}$. This resolution is simpler in certain aspects, e.g, the edge stabilizers are finitely generated and one can bound the number of parallelism classes of edges in the resolution. This result is known as Dunwoody's Lemma (\cite[Lemma 4.4]{Dun85}).

Sageev's seminal work on ends of group pairs (see \cite{Sag95}) demonstrated how CAT(0) cube complexes could be used to generalize known results about group actions on trees.
In this paper we aim to generalize Dunwoody's ideas to the realm of cube complexes.

In Section \ref{resolutions}, we construct resolutions for cube complexes and prove the following. 

\begin{theorem}\label{1.1}
Let $\gp G$ be a finitely presented group acting on a $d$-dimensional \CCC $\CC X$. There exists a $d$-dimensional \CCC and a $G$-equivariant map $F:\CC{X}'\to \CC{X}$ with the following properties:
\begin{itemize}
\item The hyperplane stabilizers in $\CC{X}'$ are finitely generated. 
\item Cube fixators and hyperplane stabilizers in $\CC{X}'$ are contained in those of $\CC X$.
In particular, if the action $\gp G \actson \CC{X}$ is free or proper
then so is $\gp G \actson \CC{X}'$.
\end{itemize}
If moreover $d\le 2$ then the action of $\gp G$ on $\CC{X}' $ is cocompact.
\end{theorem}

In order to construct the resolution we use a $d$-dimensional analogue of Dunwoody's patterns, called $d$-patterns, which we define in Section \ref{preliminaries}. 

In Section \ref{patterns} we restrict our attention mainly to square complexes, and obtain the following analogue of Dunwoody's Lemma.

\begin{theorembis}{A} \label{main thm}Let $\simp K$ be a 2-dimensional simplicial complex. Then there exists a constant $C$, depending only on $\simp K$, such that any $2$-pattern $\ptrn P$ on $\simp K$ has at most $C$ parallelism classes of tracks.
\end{theorembis}

From the above the following theorem is an immediate corollary.

\begin{theorem} \label{thm: bound on resolution} Let $\gp{G}$ be a finitely presented group. There exists a constant $C$ depending only on $G$ such that for every $\gp{G}$-action on a 2-dimensional \CCC $\CC{X}$ there exists a 2-dimensional, CAT(0) cube complex $\CC{X}'$ with the properties of Theorem \ref{1.1} and at most $C$ parallelism classes of hyperplanes. \qed
\end{theorem}

In Section \ref{surfaces}, we turn to surfaces and 3-manifolds and prove the following 2-dimensional analogue of Haken's theorem.

\begin{theorem} \label{thm: bound on submanifolds} Let $\mfld M ^n$ be an $n$-dimensional ($n=2,3$) compact manifold. There exists a constant $C$, depending only on $\mfld M$, such that if $\ptrn S$ is a collection of non-homotopic, $\pi_1$-injective, co-dimension-1, 2-sided, embedded sub-manifolds, such that the size of a pairwise intersecting collection of lifts to $\uc{\mfld M}$ is at most 2, then $|\ptrn S|\le C$.
\end{theorem}

We note that the question whether there is a bound, depending only on the dimension $d$ and the simplicial complex $\simp K$, for the number of parallelism classes of tracks in $d$-patterns is still open. An affirmative answer would imply analogous corollaries as above. 

We also note that the bound on the resolution was originally used to prove Dunwoody's Accessibility Theorem, but was also used in \cite{BeFe91} by Bestvina and Feighn to bound the number of edge-orbits in a reduced graph of group decompositions over small groups. We hope that our results will lead to analogous results for CAT(0) square complexes.

\paragraph*{Acknowledgements.}

The authors would like to thank Michah Sageev for his support and helpful comments.
The second author is also thankful for fruitful discussions with Mladen Bestvina and Dani Wise.


\section{Preliminaries }\label{preliminaries}

\subsection{CAT(0) cube complexes}
We begin by a short survey of definitions concerning CAT(0) cube complexes. For further details see, for example, \cite{Saa12}.

A \emph{cube complex} is a collection of euclidean cubes of various dimensions in which subcubes have been identified isometrically. 

A simplicial complex is \emph{flag} if every $(n+1)$-clique in its 1-skeleton spans a $n$-simplex.
A cube complex is \emph{non-positively curved} (NPC) if the link of every vertex is a flag simplicial complex. It is a \emph{\CCC} if moreover it is simply connected.

A cube complex $\CC{X}$ can be equipped with two natural metrics, the euclidean and the $L^1$-metric. With respect to the former $\CC{X}$ is NPC if and only if it is NPC \`{a} la Gromov (see \cite{Gro87}). While the latter is more natural to the combinatorial structure of CAT(0) cube complexes described below.

Given a cube $\CCc{C}$ and an edge $\CCc{e}$ of $\CCc{C}$. The midcube of $\CCc{c}$ associated to $\CCc{e}$ is the convex hull of the midpoints of $\CCc{e}$ and the edges parallel to $\CCc{e}$.
A \emph{hyperplane} is associated to $\CCc{e}$ is the smallest subset containing the midpoint of $\CCc{e}$ and such that if  contains a midpoint of an edge it contains all the midcubes containing it.
Every hyperplane $\hyp{h}$ in a \CCC $\CC{X}$ separates $\CC{X}$ into exactly two components (see \cite{NiRe98}) called \emph{halfspaces} the associated to $\hyp{h}$. A hyperplane can thus also be abstractly viewed as a pair of complementary halfspaces. 
The \emph{carrier} $N(\hyp{h})$ of $\hyp{h}$ is the union of the cubes intersecting $\hyp{h}$.
For a \CCC $\CC{X}$ we denote by $\Hyp{H}=\Hyp{H}(X)$ the set of all hyperplanes in $\CC{X}$, and by $\Hs{H}=\Hs{H}(X)$ the set of all halfspaces. For each halfspace $\hs{h}\in \Hs{H}$ we denote by $\comp{\hs{h}}\in\Hs{H}$ its complementary halfspace, and by $\hyp{h}\in\Hyp{H}$ its bounding hyperplane, which we also identify with the pair $\{\hs{h},\comp{\hs{h}}\}$.

A hyperplane in a \CCC \emph{separates} two points if each one belongs to a different halfspace. Conversely two hyperplanes are \emph{separated} by a point if there is no inclusion relation between the two halfspaces containing the point.
If two hyperplanes $\hyp h$ and $\hyp k$ intersect, we write $\hyp h \pitchfork \hyp k$.

The \emph{interval} between two vertices $\CCv{x}$ and $\CCv{y}$ of a \CCC is the maximal subcomplex $\inter{\CCv{x}}{\CCv{y}}$ contained in every halfspace containing $\CCv{x}$ and $\CCv{y}$.
Equivalently it can be seen as the union of all $L^1$-geodesics between $\CCv{x}$ and $\CCv{y}$.

Every interval of a $d$-dimensional \CCC admits an $L^1$-embedding into $R^d$ (see \cite{BCG09}).
A hyperplane intersects the interval $\inter{\CCv{x}}{\CCv{y}}$ if and only if it separates $\CCv{x}$ and $\CCv{y}$, and a cube belongs to the interval if every of its hyperplane separates them.
\subsection{Pocsets to \CCC}\label{pocsetstoccc}

We adopt Roller's viewpoint of Sageev's construction. Recall from \cite{Rol98} that a \emph{pocset} is a triple $(\Hs{P},\le,\comp{})$ of a poset $(\Hs{P},\le)$ and an order reversing involution $\comp{}:\Hs{P}\to\Hs{P}$ satisfying $\hs{h}\neq\comp{\hs{h}}$ and $\hs{h}$ and $\comp{\hs{h}}$ are incomparable for all $\hs{h}\in\Hs{P}$.

The set of halfspaces $\Hs{H}$ of a \CCC has a natural pocset structure given by inclusion relation, and the complement operation $\comp{}$. Roller's construction starts with a locally finite pocset $(\Hs{P},\le,\comp{})$ of finite width  (see \cite{Saa12} for definitions) and constructs a \CCC $\CC{X}(\Hs{P})$ such that $(\Hs{H}(X),\subseteq, \comp{})=(\Hs{P},\le,\comp{})$. We briefly recall the construction, for more details see \cite{Rol98} or \cite{Saa12}.

An \emph{ultrafilter} $U$ on $\Hs{P}$ is a subset verifying $\# \left(U\cap \left\{\hs{k},\comp{\hs{k}}\right\}\right) = 1$ for all $\hs{k}\in \Hs{P}$ and such that for all $\hs{h} \in U$, if $\hs{h} \leq \hs{k}$ then $\hs{k}\in U$. If we denote $\Hyp{P}=\left\{ \{ \hs{h},\comp{\hs{h}} \} \middle| \hs{h}\in\Hs{P} \right\}$, then $U$ can be viewed as a choice function $U:\Hyp{P}\to\Hs{P}$. Throughout the paper we will use both viewpoints.

An ultrafilter $U$ satisfies the \emph{Descending Chain Condition} (DCC) if any descending chain $\hs{k}_1 \supset \hs{k}_2 \supset \dots\supset \hs{k}_n \supset \dots$  of element of $U$ has finite length. The vertices of $\CC{X}(\Hs{P})$ are the DCC ultrafilters of $\Hs{P}$.

Two halfspaces are \emph{compatible} if their intersection is not empty in the cube complex. A subset of $\Hs{H}$ is an ultrafilter if and only if its halfspaces are pairwise compatible and it is maximal for this property.

\subsection{Patterns}\label{patternsdef}
In this section we introduce patterns. We adopt a somewhat similar definition for tracks as in  \cite{Dun85}, but we allow tracks to intersect, under some restrictions, to form $d$-patterns.

\begin{definition}
A  \emph{drawing} on a $2$-dimensional simplicial complex $\simp{K}$ is a non empty union of simple paths in the faces of $\simp{K}$ such that:
\begin{enumerate}
\item on each face there is a finite number of paths,
\item the two endpoints of each path are in the interior of distinct edges,
\item the interior of a path is in the interior of a face,
\item no two paths in a face have a common endpoint,
\item if a point $x$ on an edge $\simpe{e}$ is an endpoint then in every face containing $\simpe{e}$ there exists a path having $x$ as an endpoint.
\end{enumerate}

A \emph{pre-track} is a minimal drawing. A pre-track is \emph{self-intersecting} if it contains two intersecting paths.

Denote by $\usimp{K}$ the universal universal cover.

\begin{itemize}
\item A pre-track is a \emph{track} if none of its pre-track lifts in $\usimp{K}$ is self-intersecting.
\item A \emph{pattern} is a set of tracks whose union is a drawing.
\item A \emph{$d$-pattern} is a pattern such that the size of any at collection lifts of its tracks in $\usimp{K}$ that pairwise intersect is at most $d$.
\end{itemize}
\end{definition}

We will sometimes view a pattern as the unions of its tracks in $\simp{K}$.

\subsection{The coarse and fine pocset structures associated to a pattern} 
Let $\uptrn{P}$ be a pattern on a simply connected 2-simplex $\usimp{K}$.
For each track $\utrk{t}$ of $\uptrn{P}$, the set $\usimp{K}\setminus \utrk{t}$ has two connected components $\fhs{h}_{\utrk{t}}$ and $\comp{\fhs{h}_{\utrk{t}}}$ (see \cite{Dun85}). 
We call these components the \emph{fine halfspaces defined by $\utrk{t}$}, and the collection of all fine halfspaces is denoted by $\fHs{H}=\fHs{H}(\ptrn{P})$. 
This collection forms a locally finite pocset with respect to inclusion and complement operation $\comp{}$. 
If moreover $\uptrn{P}$ is a $d$-pattern, then $\fHs{H}$ has finite width. We denote by $\fCC{X}=X(\fHs{H})$ the \CCC constructed from the pocset $\fHs{H}$. 
Note that the dimension of $\fCC{X}$ is at most $d$. By this definition we clearly have a bijective map sending $\utrk{t}\in\ptrn{P}$ to the hyperplane $\{\fhs{h}_{\utrk{t}},\comp{\fhs{h}_{\utrk{t}}}\} \in \fHyp{H}=\Hyp{H}(\fCC{X})$.

We can also define the two \emph{coarse halfspaces defined by $\utrk{t}$} as the intersection $\chs{h}_{\utrk{t}} = \usimp{K}^0\cap \fhs{h}_{\utrk{t}}$ and $\comp{\chs{h}_{\utrk{t}}} = \usimp{K}^0\cap \comp{\fhs{h}_{\utrk{t}}}$ which are complementary in $\usimp{K}^0$. The collection of all coarse halfspaces is denoted by $\cHs{H}=X(\cHs{H})$. As above this set carries a locally finite pocset structure given by inclusion and complementation. As before, if moreover $\uptrn{P}$ is a $d$-pattern, then $\cHs{H}$ has finite width, and we denote by $\cCC{X}=X(\cHs{H})$ the \CCC constructed from the pocset $\cHs{H}$. Note that the dimension of $\fCC{X}$ is at most $d$. 

We call a connected component $A$ of $\simp{K}\setminus\ptrn{P}$ a \emph{region} of the pattern. We define the \emph{principal} ultrafilter corresponding to the region $A$ to be the set $U_A = \left\{\fhs{k}\in\fHs{H} | A \subseteq \fhs{k} \right\}$. Note that every principal ultrafilter verifies the DDC condition. Thus defines a map from the set of regions to $\fCC{X}^0$, and in particular a map from $\simp{K}^0$ to $\fCC{X}^0$. The same way we define a map from $\simp{K}^0$ to $\cCC{X}^0$.

Let $\Hsmap{\phi}:\fHs{H}\to\cHs{H}$ be the natural map sending $\fhs{h}_{\utrk{t}}$ to $\chs{h}_{\utrk{t}}=\usimp{K}^0\cap \fhs{h}_{\utrk{t}}$. The map $\Hsmap{\phi}$ respects the pocset structure and thus defines a map $\Hypmap{\phi}:\fHyp{H}\to\cHyp{H}=\Hyp{H}(\cCC{X})$.

\begin{definition}[parallelism]
Two tracks of a pattern are \emph{parallel} if they define the same coarse halfspaces. In other words if they have the same image under the map $\Hypmap{\phi}$.
\end{definition}

We have a natural map from the vertices (seen as ultrafilters) of $\cCC{X}$ to the ones of $\fCC{X}$.
Indeed the pullback of an ultrafilter by the map $\Hsmap{\phi}$ is also an ultrafilter.  Thus, we can define the map 
$\Phi^{(0)} : X^{(0)}(\cHs{H}) \rightarrow  X^{(0)}(\fHs{H})$ by $ \Phi^{(0)}(\CCv{x})= \Hsmap{\phi}^{-1}(\CCv{x})$.

\begin{proposition}
The map $\Phi^{(0)}$ can be extended to  a canonical map $\Phi : \cCC{X} \to \fCC{X}$.

Moreover, if a group $G$ acts on $\simp{K}$ leaving the pattern $\ptrn{P}$ invariant, then $G$ acts naturally on $\cCC{X}$ and $\fCC{X}$ and the map $\Phi$ is $G$-equivarient.
\end{proposition}

\begin{proof}

By construction, if two vertices $\CCv{x}$ and $\CCv{x}'$of $\cCC{X}$ are separated by the set of hyperplanes $\Hs{S}$ then the set of hyperplane separating $\Phi^{(0)}(\CCv{x})$ and $\Phi^{(0)}(\CCv{x}')$ is $\Hypmap{\phi}^{-1}(\Hs{S})$. 

If two hyperplanes $\chyp{h}_{\trk{t}}$ and $\chyp{h}_{\trk{s}}$ cross, then  $\fhyp{h}_{\trk{t}}$ and $\fhyp{h}_{\trk{s}}$ cross.

Thus given two opposite vertices in a cube $\cCCv{x}$ and $\cCCv{x'}$ separated by $n$ pairwise intersecting hyperplanes 
$\left\{\chyp{h}_1,\dots \chyp{h}_n\right\}$ in $\cCC{X}$, the interval $\inter{\Phi^{(0)}(\cCCv{x})}{\Phi^{(0)}(\cCCv{x'})}$ is isometric to
 the product cube complex $$X\left(\Hypmap{\phi}^{-1}\left(\left\{\chyp{h}_1,\dots \chyp{h}_n\right\}\right)\right) = X\left(\Hypmap{\phi}^{-1}\left(\chyp{h}_1\right)\right)\times \dots\times 
 X\left(\Hypmap{\phi}^{-1}\left(\chyp{h}_n\right)\right).$$

Given a hyperplane $\chyp h$ we define the map $$\psi_{\chyp h} : X\left(\left\{\chyp h\right\}\right) \to X\left(\Hsmap{\phi}^{-1}\left(\left\{\chyp h\right\}\right)\right)$$ sending the vertices (seen as ultrafilters) $\left\{ \chs h \right\}$  and $\left\{ \comp{\chs h}\right\}$ to $\Hsmap{\phi}^{-1}(\chs h)$ and $\Hsmap{\phi}^{-1}(\comp{\chs h})$ and the edge $X\left(\left\{\chyp h\right\}\right)$ to the $CAT(0)$ geodesic in between the endpoints of its image.

Given $n$ intersecting hyperplanes $\Hyp{K} = \left\{\chyp{h}_1,\dots \chyp{h}_n\right\}$, we define the map  $\psi_{\Hyp{K}} = \psi_{\chyp{h}_1}\times\dots \times \psi_{\chyp{h}_n}$ from the cube  $X\left(\Hyp{K}\right)$ to the product 
$$X\left(\Hsmap{\phi}^{-1}\left(\Hyp{K}\right)\right) = X\left(\Hypmap{\phi}^{-1}\left(\left\{\chyp{h}_1\right\}\right)\right)\times \dots\times X\left(\Hypmap{\phi}^{-1}\left(\left\{\chyp{h}_n\right\}\right)\right).$$
 
 We now define $\Phi$ from the cube $\inter{\cCCv{x}}{\cCCv{x'}}$ to $\inter{\Phi^{(0)}(\cCCv{x})}{\Phi^{(0)}(\cCCv{x'})}$ as the following composition:
 
 $$ \inter{\cCCv{x}}{\cCCv{x'}}\xrightarrow{\sim} X(\Hyp{H}(\inter{\cCCv{x}}{\cCCv{x'}})) \xrightarrow{\psi_{\Hyp{K}}} X(\Hsmap{\phi}^{-1}(\Hyp{H}(\inter{\cCCv{x}}{\cCCv{x'}})))\xrightarrow{\sim} \inter{\Phi^{(0)}(\cCCv{x})}{\Phi^{(0)}(\cCCv{x'})}. $$

It is straight forwoard to verify that this extends the map $\Phi^{(0)}$.

The $G$-equivariance follows from the canonicity of the map.

\end{proof}

\section{Resolutions}
\label{resolutions}

Let $\gp{G}$ be a finitely presented group. Let $\simp{K}$ be a fixed finite
triangle complex such that $\gp G \simeq \pi_{1}(\simp K,\simpv{v}_{0})$ for some $\simpv{v}_{0}\in \simp K$.
Let $\{ \usimpv{v}_{0},\ldots,\usimpv{v}_{l}\} $, $\{ \usimpe{e}_{0},\ldots,\usimpe{e}_{m}\} $
and $\{ \usimpf{f}_{1},\ldots,\usimpf{f}_{n}\} $ be sets
of representatives for the $\gp G$-orbits of 0-, 1- and 2-cells in $\usimp{K}$.

Now, let $\gp G$ act on a $d$-dimensional CAT(0) cube complex $\oCC{X}$.
For each $\usimpv{v}_{i}$ choose a vertex $\oCCv{x}_{i}$ in $\oCC X$. Since
$\gp G$ acts freely on $\usimp{K}$, one can extend the map $\usimpv{v}_{i}\mapsto \oCCv{x}_{i}$
to a $\gp G$-equivariant map $ f : \usimp{K}^{(0)}\to \oCC{X}^{(0)}$ by sending
$\gpelt{g}\usimpv{v}_{i}$ to $\gpelt{g}\oCCv{x}_{i}$. We extend this map to a map on the 1-skeleton
of $\usimp{K}$ by sending each edge representative $\usimpe{e}_{i}$
linearly to a combinatorial geodesic connecting the images of its endpoints $f(i(\usimpe{e}_{i}))$
and $f(t(\usimpe{e}_{i}))$ in $\oCC X$, and extend $\gp G$-equivariantly
to $\usimp{K}^{(1)}$. Similarly, extend the map to a $\gp G$-equivariant
map $f:\simp{K}\to \oCC{X}$ by sending the 2-cells $\usimpf{f}_{i}$ to disks whose
boundary is $f(\partial\usimpf{f}_{i})$, such exist because $\oCC X$ is
simply connected. We may further assume that the image of $f$ is
in $\oCC{X}^{(2)}$, the 2-skeleton of $\oCC{X}$, it is transverse to the hyperplanes
and has minimal number of squares of $\oCC{X}$. Such a map is called a \emph{minimal disk}. For more details see \cite{Sag95}.

\begin{lemma}
Let $\uptrn{P}=\bigcup_{\ohyp{h}\in\oHyp{H}}f^{-1}(\ohyp h)$
be the pullback of the hyperplanes of $\oCC X$ to $\usimp{K}$. The set $\uptrn{P}$ induces a $d$-pattern $\ptrn P$ on $\simp K$.
\end{lemma}
\begin{proof}
Note that the pullback of each $\ohyp{h}\in\oHyp{H}$ defines a 1-pattern on $\usimp{K}$ (see \cite{DiDu89}). 

We are left to show that the size of a collection of pairwise crossing tracks is at most d. Let $\utrk{t}_{1},\ldots,\utrk{t}_{k}$ be distinct pairwise intersecting tracks in $\uptrn{P}$.
Each $\utrk{t}_{i}$ maps into the corresponding hyperplane $\ohyp{h}_{i}$.
By the transversality to the hyperplanes, the hyperplanes $\ohyp{h}_{1},\ldots,\ohyp{h}_{k}$
are distinct intersecting hyperplanes. Thus $k\le d$.
\end{proof}

The $d$-pattern defines the fine cube complex $\fCC{X}$ (see Section \ref{preliminaries}) on which $\gp G$ acts. Note that the map $f$ induces
a map, which we denote by $\Hsmap{f}$, from $\fHs{H}$, the set
of halfspaces of $\fCC{X}$, to $\oHs{H}$. This map respects the
complementation operation, thus defines a map $\Hypmap{f}:\fHyp{H}\to\oHyp{H}$
on hyperplanes. Note also that the image of $\Hypmap{f}$ consists
of all the hyperplanes that divide non trivially the image of $f(\simp K)$.

\begin{proposition}
There exists a $\gp G$-equivariant combinatorial map $F:\fCC{X}\to \oCC{X}$, which induces the map $\Hsmap{f}$ on halfspaces.
\end{proposition}
\begin{proof}
Let us first define the map $F$ on the vertices of $\fCC{X}$. Let $\fCCv{x}$
be a vertex of $\fCC{X}$, i.e a DCC ultrafilter on the halfspaces, which we regard as the map $\fCCv{x}:\fHyp{H}\to\fHs{H}$ which assigns to each hyperplane its halfspace that contains $\fCCv{x}$.
We define $F(\fCCv{x})$ as follows, for each $\ohyp{h}\in\oHyp{H}$
either $\ohyp{h}$ belongs to the image of $\Hypmap{f}$ or not. In
the former case, let $\fhyp{h}\in\fHyp{H}$ be a minimal hyperplane with respect to $\fCCv{x}$
among $\left(\Hypmap{f}\right)^{-1}(\ohyp{h})$, and define
$F(\fCCv{x})(\ohyp{h})=\Hsmap{f}(\fCCv{x}(\fhyp{h}))$. In the latter case we choose
$\oCCv{x}(\ohyp{h})$ to be the halfspace $\ohs{h}$ which contains $f(\simp K)$.

The map $F$ does not depend on the choice of $\fhyp{h}$: for every $\ohyp{h}\in\oHyp{H} $ and every  $\fhyp{h}_{1},\fhyp{h}_{2}\in\Hypmap{f}^{-1}(\ohyp{h})$, minimal with respect to $\fCCv{x}$ among $\Hypmap{f}^{-1}(\ohyp{h})$, we have $\Hsmap{f}(\fhyp{h}_{1})=\Hsmap{f}(\fhyp{h}_{2})$. Otherwise a geodesic path that connects them in $\usimp{K}$ is mapped by $f$
to a path that passes from $\ohs{h}$ to $\comp{\ohs h}$ in $\oCC X$ without crossing
$\ohyp{h}$.

The map $F(\fCCv{x}):\oHyp{H}\to\oHs{H}$ is an ultrafilter:
let $\ohyp{h},\ohyp{k}\in\oHyp{H}$ which do not cross, let
$\fhyp{h},\fhyp{k}$ be the hyperplanes in the definition of $F(\fCCv{x})$,
and let $\fhs{h},\fhs{k}$ be their orientation in $\fCCv{x}$. Clearly the orientation
of the halfspaces $\fhs{h},\fhs{k}\subset \usimp K$ is such that they have a non trivial intersection (otherwise, $\fCCv{x}$ is not an ultrafilter). If $p$ is a point in this intersection then both $F(\fCCv{x})(\ohyp{h}),F(\fCCv{x})(\ohyp{k})$ contain $f(p)$, showing that they form a compatible pair.

The map $F:\fCC{X}^{(0)}\to \oCC{X}^{(0)}$ extends to $F:\fCC{X}\to \oCC{X}$: if $\fCCv{x}_{1},\ldots,\fCCv{x}_{2^{k}}$ are the vertices of a $k$-dimensional
cube in $\fCC{X}$, i.e their ultrafilters differ on exactly $k$ hyperplanes $\fhyp{h}_{1},\ldots,\fhyp{h}_{k}$.
Then there images will differ exactly on the collection of distinct
pairwise transverse hyperplanes $\Hypmap{f}(\fhyp{h}_{1}),\ldots,\Hypmap{f}(\fhyp{h}_{k})$.\end{proof}

The pair $(\fCC{X},F)$ is called a \emph{fine resolution} of $\oCC{X}$. In Section \ref{preliminaries}, we also constructed the coarse \CCC $\cCC{X}$ and a $\gp G$-equivariant map $\Phi: \cCC{X}\to\fCC{X}$ between the two complexes. The pair $(\cCC{X},F\circ \Phi)$ is called a \emph{coarse resolution} of $\oCC{X}$. Note that both resolution depend on the choice of $\simp K$ and on the equivariant map $f:\simp{K}\to\oCC{X}$.

\begin{proposition}\label{properties of the fine resolution}
The fine resolution $(\fCC{X},F)$ has the following properties.
\begin{itemize}
\item The hyperplane stabilizers in $\fCC{X}$ are finitely generated. 
\item Cube and hyperplane stabilizers in $\fCC{X}$ are contained in those of $\oCC X$.
In particular, if the action $\gp G \actson \oCC{X}$ is proper or free
then so is $\gp G \actson \fCC{X}$.
\end{itemize}
\end{proposition}

\begin{proof}

The stabilizer of a track $\utrk t \in \uptrn P$ is the image of $\pi_1(\trk t)$ in $\pi_1 (\simp K) \simeq G$ under the inclusion map. The track $\trk{t}$ is a finite graph in $\simp K$, and thus finitely generated.

The map $F:\fCC{X} \to \oCC X$ is $\gp G$-equivariant and combinatorial, thus for all cube $\fCCc{C} \in \fCC{X}$ we have $\Stab_{\gp{G}} (\fCCc{C}) < \Stab_{\gp G} ( F(\fCCc C) )$, and similarly for hyperplanes.
\end{proof}

Since the map $\Phi:\cCC{X}\to\fCC{X}$ is not combinatorial, the properties of the coarse resolution are slightly weaker.

\begin{proposition}
The coarse resolution $(\cCC{X},F\circ \Phi)$ has the following properties.
\begin{enumerate}
\item The hyperplane stabilizers in $\cCC{X}$ are finitely generated. 
\item If the hyperplanes in $\oCC{X} / G$ are embedded, i.e, for all $g\in\gp{G}$ and $\ohyp{h}\in\oHyp{H}$ either $g\ohyp{h}=\ohyp{h}$ or $g\ohyp{h}\cap\ohyp{h}=\emptyset$, then the oriented hyperplane stabilizers in $\cCC{X}$, i.e. stabilizers of the hyperplane which do not exchange the halfspace, are contained in those of $\oCC{X}$.
\item Cube fixators in $\cCC{X}$ are contained in fixators of cubes of the same dimension in $\oCC{X}$.
\item Cube stabilizers in $\cCC{X}$ act elliptically on $\oCC{X}$.
\item If the action $\gp G \actson \oCC{X}$ is proper or free then so is $\gp G \actson \cCC{X}$.
\end{enumerate}
\end{proposition}

\begin{proof}
Note that the implications $3\implies 4\implies 5$ are trivial.

Recall from Section \ref{preliminaries} that the preimage $\Hypmap{\phi}^{-1} (\chyp{h}_{\trk{t}})$ is finite, and thus the stabilizer of $\chyp{h}_{\trk{t}}$ is generated by the stabilizers of the tracks $\Hypmap{\phi}^{-1} (\chyp{h}_{\trk{t}})$ and finitely many elements permuting this collection. This completes the proof of 1.

If the hyperplanes in  $\oCC{X} / G$ are embedded then so are the hyperplanes in $\cCC{X} / G$ and in $\fCC{X} / G$, otherwise there is an element $g\in\gp{G}$ and a track $\utrk{t}\in\uptrn{P}$ such that  $\utrk{t}$ and $g \utrk{t}$ cross, which by the map $f$ would imply that $\Hypmap{f}(\fhyp{h}_{\utrk{t}})$ and $g\Hypmap{f}(\fhyp{h}_{\utrk{t}})$ cross.

Also note that the collection of $\Hsmap{\phi}^{-1} (\chs{h})$ is a poset on which the stabilizer of $\chs{h}$ acts. Moreover, two hyperplanes in $\Hsmap{\phi}^{-1} (\chs{h})$ cross if and only if they are incomparable. Since no hyperplane can be sent to a crossing hyperplane, the stabilizer of $\chs{h}$ fixes this poset. Thus the stabilizer of $\chs{h}$ is included in the stabilizer of each of $\fhs{h}\in\Hsmap{\phi}^{-1} (\chs{h})$. This completes the proof of 2.

By the construction of the map $\Phi:\cCC{X}\to\fCC{X}$ we see that each cube $\cCCc{C}$ of $\cCC{X}$ is affinely embedded into a product region of higher dimension than that of $\cCCc{C}$, thus a generic point in $\cCCc{C}$ is sent to a generic point of a cube of higher dimension. This completes the proof of 3.
\end{proof}

Though one might expect that the resolution of a cocompact $\gp G$-action would be cocompact, this is not always the case as the following example shows.
Let $\gp G = \Z ^2 =\left< e_1=(1,0),e_2=(0,1) \right>$, with the presentation complex $\simp K$  obtained by gluing two triangles along an edge to form a square and then identifying opposite edges to form a torus. The group $\gp G$ acts on $\oCC X = \R ^2 \times [0,1]$ (with the standard cubulation by unit cubes) by $e_i(x,t)=(e_i+x,1-t)$ for $i=1,2$, i.e it acts by translations on the first factor and by inversions on the second.

The pattern obtained on $\usimp K=\R ^2$  consists of 3 infinite sets of tracks of different parallelism  classes of lines (see Figure \ref{non-example}). Therefore, the associated CAT(0) cube complexes $\cCC{X}$ (in this case $\fCC{X}=\cCC{X}$) is the standard cubulation of $\R^3$, on which $\gp{G}=\Z^2$ does not act cocompactly.

\begin{figure}[!ht]
\begin{center}
 \def\svgwidth{\textwidth}
 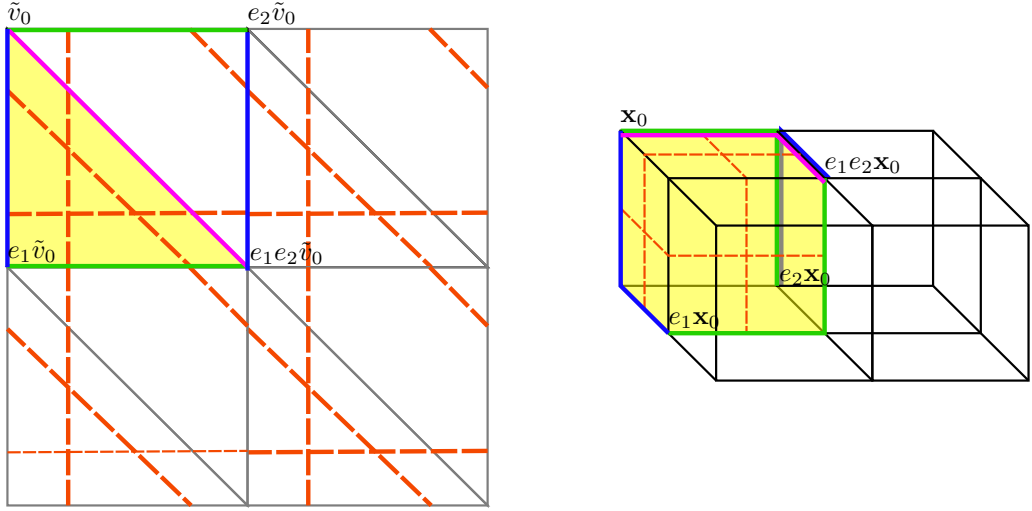
 \end{center}
 \caption{An example for a non-cocompact resolution of a cocompact action}
 \label{non-example}
\end{figure}

However, such an example cannot occur in dimension 2 (or smaller). In fact an even stronger statement holds in this case.

\begin{proposition}\label{cocompactness of 2d resolutions}
Let $\simp K$ be a compact triangle complex, and let $\gp G = \pi _1 (\simp K)$. If $\gp G \actson \oCC X$ , a 2-dimensional \CCC , then $\gp G \actson \fCC{X}$ cocompactly. 

\end{proposition}
\begin{proof}
By assumption $\gp G$ acts cocompactly on $\usimp K$, and the pattern on each triangle of $\simp K$ is finite. Thus, $\gp G$ acts cocompactly on the set of regions, i.e the set of connected components of $\simp{K}\setminus \ptrn{P}$. Hence, it is enough to show that every vertex $\fCCv{x}\in\fCC{X}$  is a principal ultrafilter, i.e corresponds to a region in $\usimp K$.

Let $\fHs{H}_{\fCCv{x}}$ be the collection of minimal halfspaces in $\fCCv{x}$. It is enough to show that the intersection $A_\fCCv{x}=\bigcap_{\fhs{h}\in \fHs{H}_{\fCCv{x}}} \fhs{h}\subseteq \usimp{K}$ is non-empty since the intersection is a region in $\usimp{K}$ corresponding to $\fCCv{x}$.
Fix $\fhs{h}_{\utrk{t}} \in \fHs{H}_{\fCCv{x}}$. There are two cases to consider. 

Case 1. The track $\utrk t$ does not intersect any other track. In this case, every other $\fhs{h}_{\utrk{v}} \in \fHs{H} _{\fCCv{x}}$  contains $\utrk{t}$, for otherwise $\fhs{h}_{\utrk{t}}$ is not minimal. Thus, any point in $\fhs{h}_{\utrk{t}}\subseteq \usimp{K}$ close enough to $\utrk{t}$ will be in the intersection above.

Case 2. There exists $\fhs{h}_{\utrk{u}} \in \fHs{H}_{\fCCv{x}}$ such that $\utrk{t}$ and $\utrk{u}$ intersect. In this case, for all $\fhs{h}_{\utrk{v}} \in \fHs{H}_{\fCCv{x}}$ the track $\utrk{v}$ cannot intersect both $\utrk{t}$ and $\utrk{u}$. Hence, the corresponding fine halfspace $\fhs{h}_{\utrk{v}}$ contains either $\utrk{t}$ or $\utrk{u}$. Thus, any point in $\fhs{h}_{\utrk{t}} \cap \fhs{h}_{\utrk{u}}\subseteq \usimp{K}$ close enough to $\utrk{t}\cap \utrk{u}$ will be in the intersection above.
\end{proof}

Recall from \cite{DiDu89} that a group is \emph{almost finitely presented} if it acts freely, cocompactly on a simplicial complex $\bar{\simp{K}}$ with $H^1(\bar{\simp{K}},\Z/2\Z)=0$. All of the above works when replacing $\usimp{K}$ with $\bar{\simp{K}}$. Hence, by Proposition \ref{properties of the fine resolution} and Proposition \ref{cocompactness of 2d resolutions} we get the following corollary.

\begin{corollary}
Any almost finitely presented group that acts properly (resp. freely) on a 2-dimensional \CCC  acts properly (resp. freely) and cocompactly on a 2-dimen\-sional \CCC, and in particular it is finitely presented. \qed
\end{corollary}


\section{Bounding the number of tracks in a pattern}\label{patterns}

This section focuses on proving Theorem \ref{main thm} which can be formulated as follows.
\begin{theorembis}{A'}\label{patternsbound}
Let $\simp{K}$ be a $2$-dimensional simplicial complex and $\ptrn{P}$ be a $2$-pattern on $\simp{K}$ with no parallel tracks. Then there exists an integer $D$, depending only on $\simp{K}$, such that the number of tracks in $\ptrn{P}$ is bounded by $D$.
\end{theorembis}

We begin by defining a weak notion of parallelism for adjacent hyperplanes in an interval. First, recall that a pair of non-crossing hyperplanes are \emph{adjacent} if their carrier contains a common vertex. 

\begin{definition}
Given a pair of non-crossing hyperplanes $(\hyp{h}, \hyp{k})$, a \emph{parallelism obstructing pair (\blockingpair)} is a pair of crossing hyperplanes $(\hyp{h}', \hyp{k}')$ such that $\hyp{h}'\pitchfork \hyp{k}$ but $\hyp{h}'\not\pitchfork \hyp{h}$, and $\hyp{k}'\pitchfork \hyp{h}$ but $\hyp{k}'\not\pitchfork\hyp{k}$ (see Figure \ref{blockingpair}). It is a \blockingpair in an interval $I$ if the four hyperplanes intersect $\itvl{I}$.

Two non-crossing hyperplanes $(\hyp{h}, \hyp{k})$ in an interval $\itvl{I}$ are \emph{\adjP} in $\itvl{I}$, if they are adjacent and do not have a \blockingpair in $I$.
\end{definition}

\begin{figure}[!ht]
\begin{center}
 \def\svgwidth{0.5\textwidth}
 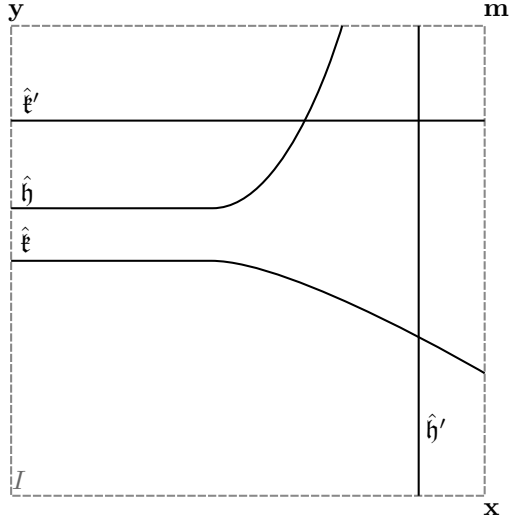
 \end{center}
 \caption{Parallelism obstructing pair (\blockingpair).}
 \label{blockingpair}
\end{figure}

In a \CCC we write $\hyp{h} <_\CCv{x} \hyp{k}$ when $\hyp{h}$ separates $\CCv{x}$ from  $\hyp{k}$.  Here, $\CCv{x}$ can be a vertex or a hyperplane. If $\CCv{x}$ is a vertex it is equivalent to say that the ultrafilter corresponding to $\CCv{x}$ satisfies $\CCv{x}(\hyp{h})<\CCv{x}(\hyp{k})$

\begin{lemma}\label{lemma1}
For every $d$, there exists $C(d)$ such that for any interval $\itvl{I}$ of a \CCC of dimension $d$ and for every vertex $\CCv{m}$ of $\itvl{I}$, the set of (non-ordered) pairs of \adjP hyperplanes separated by $\CCv{m}$ has cardinality at most $C(d)$. 
\end{lemma}

\begin{proof}
We start by proving that if a pair of \adjP hyperplanes $(\hyp{h},\hyp{k})$ in $\itvl{I}$ is separated by $\CCv{m}$, then at least one of the two hyperplane is adjacent to $\CCv{m}$.
Assume not, then there exists $\hyp{h}' <_\CCv{m} \hyp{h}$ and $\hyp{k}'<_\CCv{m} \hyp{k}$. Since the four hyperplanes belong to the interval $\itvl{I}$ and there are no facing triples in an interval, we get $\hyp{k}\pitchfork\hyp{h}'$, $\hyp{h}'\pitchfork\hyp{k}'$, and  $\hyp{k}'\pitchfork\hyp{h}$. Contradicting that $\hyp{h}$ and $\hyp{k}$ are \adjP.

Without loss of generality let $\hyp{h}$ be the hyperplane adjacent to $\CCv{m}$. Since $\itvl{I}$ is an interval of dimension $d$, thus can be embedded in $\R^d$, there are at most $2d$ hyperplanes adjacent to $\CCv{m}$ in $\itvl{I}$.

Now assume $(\hyp{h},\hyp{k})$ and $(\hyp{h},\hyp{k}')$ are \adjP and separated by $\CCv{m}$.  The hyperplanes $\hyp{k}$ and $\hyp{k}'$ cannot be facing since otherwise $\hyp{h}$, $\hyp{k}$ and $\hyp{k}'$ would be a facing triple of $\itvl{I}$.
The hyperplanes $\hyp{k}$ for which $(\hyp{h},\hyp{k})$ is an \adjP pair separated by $\CCv{m}$, pairwise cross. Hence, there are at most $d$ of them. Thus we can set $C(d)=2d^2$.
\end{proof}

\begin{lemma}\label{lemma2}
For any pair of intervals $\itvl{I}_1=\inter{\CCv{x}}{\CCv{y}_1}$ and $\itvl{I}_2 = \inter{\CCv{x}}{\CCv{y}_2}$ in a \CCC  $\CC{X}$ of dimension $2$, there are at most four pairs of hyperplanes intersecting $\itvl{I}_1$ and $\itvl{I}_2$, \adjP in $\itvl{I}_1$ but not in $\itvl{I}_2$.
\end{lemma}

\begin{proof}
Denote by $\CCv{m}$ the median point of $\CCv{x}$, $\CCv{y}_1$ and $\CCv{y}_2$. Note that every hyperplane intersecting $\itvl{I}_1$ and $\itvl{I}_2$ separates $\CCv{m}$ from $\CCv{x}$ and that no hyperplane separates $\CCv{m}$ from $\CCv{x}$ and $\CCv{y}_2$. Without loss of generality we assume throughout that a pair of \adjP hyperplanes $(\hyp{h},\hyp{k})$ is such that $\hyp{h}<_\CCv{m} \hyp{k}$ (or equivalently $\hyp{k}<_\CCv{x} \hyp{h}$).

Note that two crossing hyperplanes that intersect a common interval, cross inside the interval. Thus if a pair of hyperplanes is \adjP in $\itvl{I}_1$ but not in $\itvl{I}_2$, then at least one of the hyperplanes of the \blockingpair in $\itvl{I}_2$ does not intersect $\itvl{I}_1$. Since $\hyp{h'}<_\CCv{x} \hyp{h}$,  $\hyp{h}'$ has to intersect $\itvl{I}_1$. Therefore the hyperplane $\hyp{k}'$ separates $\CCv{y_2}$ from $\CCv{x}$ and $\CCv{m}$ (see Figure \ref{Crossing}).
\begin{figure}[!ht]
\begin{center}
 \def\svgwidth{0.7\textwidth}
 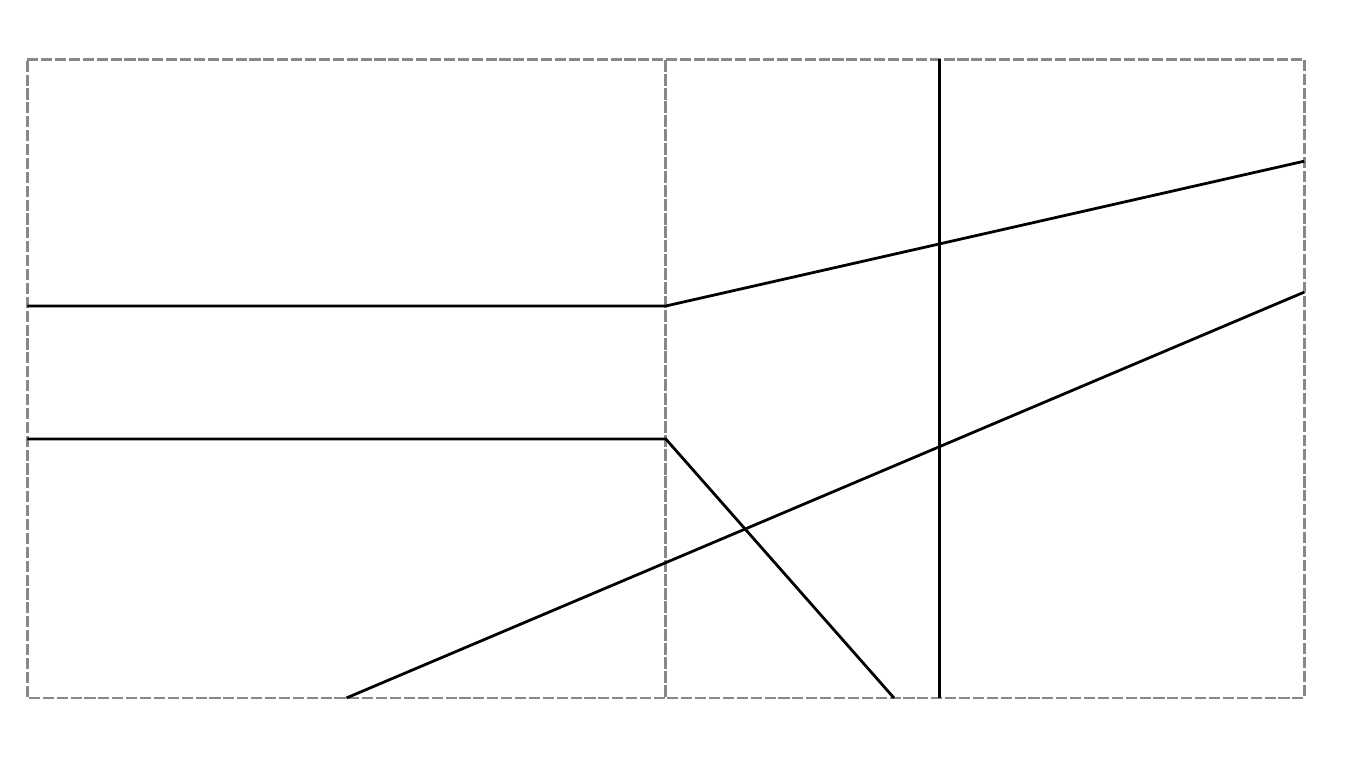
 \end{center}
 \caption{The hyperplane $\hyp{k}'$ has to separate $\CCv{m}$ and $\CCv{y}_2$}
 \label{Crossing}
\end{figure}

Suppose there exists two pairs of hyperplanes $(\hyp{h}_1,\hyp{k}_1)$ and $(\hyp{h}_2,\hyp{k}_2)$, \adjP in $\itvl{I}_1$, but not in $\itvl{I}_2$. These pairs are adjacent and non-crossing in $\itvl{I}_1$ and thus also in $\itvl{I}_2$. Therefore, there exist two \blockingpairs $(\hyp{h}'_1,\hyp{k}'_1)$ and $(\hyp{h}'_2,\hyp{k}'_2)$ for $(\hyp{h}_1,\hyp{k}_1)$ and $(\hyp{h}_2,\hyp{k}_2)$ respectively. Moreover they can be chosen such that $\hyp{k}'_1$ and $\hyp{k'}_2$ are minimal with respect to $\CCv{x}$.

We show that $\hyp{h}_1\pitchfork\hyp{h}_2$. Then as we are in dimension 2 each $\hyp{h}_i$ belong to at most $2$ pairs (the other hyperplanes have to be adjacent to $\hyp{h}_i$ so intersects). This would give the bound of $4$.

Suppose that  $\hyp{h}_1\leq_{\CCv{m}}\hyp{h}_2$.

As $\hyp{h}_1\leq_\CCv{m} \hyp{h}_2$, the interval $\inter{\CCv{m}}{\CCv{y}_2}$ and the hyperplane $\hyp{h}_2$ are separated by $\hyp{h}_2$. The hyperplane $\hyp{k}'_2$ intersects both $\inter{\CCv{m}}{\CCv{y}_2}$ and $\hyp{h}_2$ hence it crosses $\hyp{h}_1$. Thus, as $\CC{X}$ is two dimensional, $\hyp{k}'_1$ and $\hyp{k}'_2$ cannot cross. We have two cases:

Case 1. If $\hyp{k}'_2<_\CCv{m} \hyp{k}'_1$. Then the interval $\inter{\CCv{m}}{\CCv{x}}$ and the hyperplane $\hyp{k}'_1$ are separated by $\hyp{k}'_2$. Thus $\hyp{h}'_1$ and $\hyp{k}'_2$ cross. Again, because the \CCC is two dimensional, the hyperplanes $\hyp{k}'_2$ and $\hyp{k}_1$ cannot cross. Thus the pair $(\hyp{h}'_1,\hyp{k}'_2)$ is a \blockingpair for $(\hyp{h}_1,\hyp{k}_1)$, contradicting the minimality of $(\hyp{h}_2,\hyp{k}_2)$.

Case 2. If $\hyp{k}'_1\leq_\CCv{m} \hyp{k}'_2$. First notice that as $\hyp{h}_2$ intersects $\inter{\CCv{m}}{\CCv{x}}$ and $\hyp{k}'_2$, the hyperplane $\hyp{k}'_1$ crosses $\hyp{h}_2$. This implies that $\hyp{h}_2$ and $\hyp{k}_1$ are distinct.

 \begin{figure}[!ht]
\begin{center}
  \def\svgwidth{0.7\textwidth}
 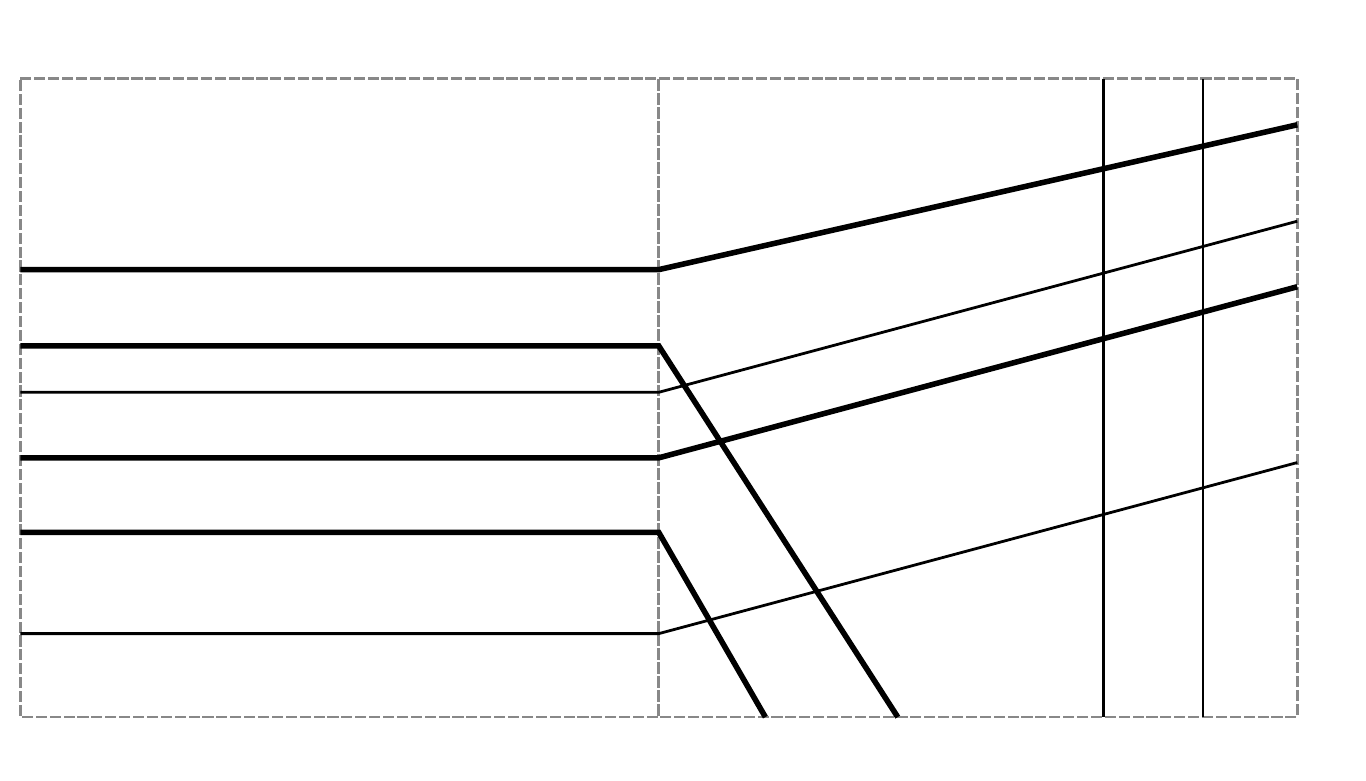
 \end{center}
 \caption{Configuration in the case $\hyp{k}'_1\leq_\CCv{m} \hyp{k}'_2$.}
 \label{Configuration2}
\end{figure}

As $\hyp{k}'_1\leq_\CCv{x} \hyp{k}'_2$ and $\hyp{k}_1<_\CCv{x} \hyp{k}'_1$, we have  $\hyp{k}_1<_\CCv{x} \hyp{k}'_2$ and $\hyp{h}'_2\pitchfork\hyp{k}_1$. Again, by dimension, the hyperplanes $\hyp{k}_1$ and $\hyp{k}_2$ do not cross.

On the one hand, the hyperplanes $\hyp{k}_1$ and $\hyp{h}_2$ cannot cross, otherwise $(\hyp{k}_1, \hyp{h}'_2)$ would be a \blockingpair for $(\hyp{h}_2,\hyp{k}_2)$, contradicting the fact that $(\hyp{h}_2,\hyp{k}_2)$ is \adjP in $\itvl{I}_1$ (see Figure \ref{Configuration2}).

On the other hand, if they do not cross, then  $\hyp{k}_1<_{\CCv{x}} \hyp{h}_2$ (they are distinct). As $\hyp{k}_1$  and $\hyp{h}_1$ are adjacent, this implies $\hyp{h}_1 =\hyp{h}_2 $. But $\hyp{k}_2$  and $\hyp{h}_2$ are adjacent, so $\hyp{k}_1$ and $\hyp{k}_2$ have to cross. 

Similarly, since $\hyp{k}_1$  and $\hyp{h}_1$ are adjacent and $\hyp{h}'_2$ crosses $\hyp{k}'_2$, the hyperplanes $\hyp{h}'_2$ and $\hyp{k}_1$ also have to cross.
But $\hyp{h}'_2$ and $\hyp{k}_2$ cross. This is impossible since the \CCC is $2$-dimensional.

\end{proof}

\begin{proof}[Proof of Theorem \ref{patternsbound}]
Let $\usimp{K}$ be the universal cover of $\simp{K}$ and $\uptrn{P}$ the pattern on $\usimp{K}$ associated to $\ptrn{P}$. Let $\cCC{X}$ be the coarse \CCC associated to the pattern $\uptrn{P}$. Since $\ptrn P$ is a $2$-pattern, the \CCC $\cCC{X}$ is a square complex.

For a vertex $\usimpv{x}$ in $\usimp{K}$ call $\cCCv{\bar{x}}$ the corresponding vertex in $\cCC{X}$. Similarly the hyperplane corresponding to a track $\utrk{t}$ in $\uptrn{P}$ is called $\chyp{h}_{\utrk{t}}$. A triangle in $\cCC{X}$ is a triplets of vertices $(\cCCv{\bar{x}}, \cCCv{\bar{y}}, \cCCv{\bar{z}})$  coming from a triangle $(\usimpv{x},\usimpv{y},\usimpv{z})$ of $\usimp{K}$.

Two tracks $\utrk{t}$ and $\utrk{t}'$ of $\uptrn P$ are  \adjP  if they cross an edge $[\usimpv{x},\usimp{y}]$ such that $\chyp{h}_{\utrk{t}}$ and $\chyp{h}_{\utrk{t}'}$ are \adjP in the interval defined by $\cCCv{\bar{x}}$ and $\cCCv{\bar{y}}$.

Note that if a pair of hyperplanes $(\chyp{h}, \chyp{k})$ in $\cCC{X}$ intersect an interval $[\cCCv{\bar{v}}, \cCCv{\bar{w}}]$ in which they are \adjP but not parallel, then:
\begin{enumerate}
\item  either there exists some triangle $(\cCCv{\bar{x}}, \cCCv{\bar{y}}, \cCCv{\bar{z}})$ such that $(\chyp{h}, \chyp{k})$ is \adjP in  $[\cCCv{\bar{x}}, \cCCv{\bar{y}}]$ but is separated by the midpoint of $(\cCCv{\bar{x}}, \cCCv{\bar{y}}, \cCCv{\bar{z}})$,
\item or  there exists some triangle $(\cCCv{\bar{x}}, \cCCv{\bar{y}}, \cCCv{\bar{z}})$ such that $(\chyp{h}, \chyp{k})$ is \adjP in  $[\cCCv{\bar{x}}, \cCCv{\bar{y}}]$, intersects $[\cCCv{\bar{x}}, \cCCv{\bar{z}}]$ but is not \adjP in it.
\end{enumerate} 

Indeed, if they are not parallel, there exists a triangle in which they are separated. Take a sequence of triangles $T_1,\dots, T_n$ such that the vertices  $\cCCv{\bar{v}}$ and $\cCCv{\bar{w}}$ are vertices of $T_1$, two consecutive triangles $T_i$ and $T_{i+1}$ share an edge $(\cCCv{\bar{v}}_i, \cCCv{\bar{w}}_i)$ crossed by $(\chyp{h}, \chyp{k})$,  and $(\chyp{h}, \chyp{k})$ are separated by a vertex of $T_n$. If it exists take the smallest $i$ such that $(\chyp{h}, \chyp{k})$ is not \adjP in the interval defined by $\cCCv{\bar{v}}_i$ and $\cCCv{\bar{w}}_i$. Then the triangle $T_i$ fits the second criterion. If such $i$ does not exists the pair $(\chyp{h}, \chyp{k})$ is \adjP in $(\cCCv{\bar{w}}_n,\cCCv{\bar{v}}_n)$, but is separated by some vertex of $T_n$ thus by the midpoint of $T_n$.

We show that for every interval $\inter{\cCCv{\bar{v}}}{\cCCv{\bar{w}}}$ and every hyperplane $\chyp h$ of $\inter{\cCCv{\bar{v}}}{\cCCv{\bar{w}}}$, either $\chyp h$ is adjacent to $\cCCv{\bar{v}}$ or there exists an other hyperplane $\chyp k$ \adjP to $\chyp h$ in $\inter{\cCCv{\bar{v}}}{\cCCv{\bar{w}}}$.

If $\chyp h$ is not adjacent to $\cCCv{\bar{v}}$ then there exists an adjacent hyperplane $\chyp k_1 <_v \chyp h$. If $\chyp k$ is not \adjP to $\chyp h$ then there exists a \blockingpair $({\hyp k}', {\hyp h}')$. We can assume that ${\hyp h}'$ is minimal with respect to $\chyp h$.  If ${\hyp h}'$ is not \adjP to $\chyp h$, there exist another \blockingpair $(\chyp h_1, {\hyp h}'_1)$. But as $\chyp k$ is minimal with respect to $\chyp h$, thus ${\hyp h}'_1\pitchfork \chyp k$. Hence, we obtain five distinct hyperplanes $\chyp h, {\hyp k}', {\hyp h}', \chyp k, {\hyp h}'_1$ which cyclically cross which is forbidden in a interval of a 2-dimensional \CCC.

Notice that in a 2-dimensional interval $\inter{\cCCv{\bar{v}}}{\cCCv{\bar{w}}}$, there are at most $2$ hyperplanes adjacent to $\cCCv{\bar{v}}$.

Since there are no parallel tracks in $\ptrn P$, the above discussion shows that a hyperplane $\chyp h$ in $\cCC{X}$  belongs to one of the following categories.

\begin{enumerate}
\item The hyperplane $\chyp h$ is adjacent to a (fixed) extremity of an edge.
\item  There exist some hyperplane $\chyp k$ and some triangle $(\cCCv{\bar{x}}, \cCCv{\bar{y}}, \cCCv{\bar{z}})$ such that $(\chyp{h}, \chyp{k})$ is \adjP in  $[\cCCv{\bar{x}}, \cCCv{\bar{y}}]$ but is separated by the midpoint of $(\cCCv{\bar{x}}, \cCCv{\bar{y}}, \cCCv{\bar{z}})$.
\item There exist some hyperplane $\chyp k$ and some triangle $(\cCCv{\bar{x}}, \cCCv{\bar{y}}, \cCCv{\bar{z}})$ such that $(\chyp{h}, \chyp{k})$ is \adjP in  $[\cCCv{\bar{x}}, \cCCv{\bar{y}}]$, intersects $[\cCCv{\bar{x}}, \cCCv{\bar{z}}]$ but is not \adjP in it.
\end{enumerate}

By lemma \ref{lemma1} each midpoint of triangle separates at most $3\times 8$ pairs of \adjP hyperplanes.
By lemma \ref{lemma2}, each triangle contains at most $24$ pairs of hyperplanes \adjP with respect to one edge but not with respect with another.
For each edge there is at most $2$ hyperplane adjacent to a (fixed) extremity.

If we denote by $E$ and $T$ be the numbers of edges and triangles respectively of $\simp K$,
then there are at most $2\times(8\times 3 + 24) \times T + 2 \times E = 96T + 2E$  orbits of hyperplanes in $\cCC{X}$ under the action of $\pi_1(\simp K)$.

Therefore, the pattern $\ptrn P$ contains at most $96T + 2E$ tracks.
\end{proof}


\section{Bounding submanifolds in surfaces and 3-manifolds}\label{surfaces}

\subsection{Patterns on surfaces} 

Let $\mfld S$ be a compact surface, with a triangulation $\simp K \homeo \mfld S$. Let $\ptrn P$ be a collection of non-homotopic essential, 2-sided, properly immersed curves and arcs such that the size of a pairwise intersecting collection of their lifts to $\uc{\mfld S}$ is at most $d$. We would like to homotope the pattern $\ptrn P$ such that $\ptrn P$ will be a $d$-pattern in $\simp K$.

By homotoping each curve to the corresponding geodesic curve in $\mfld S$, one assures that the lift to the universal cover of each curve is an embedded line, while still having that the size of a pairwise intersecting collection of their lifts to $\uc{\mfld S}$ is at most $d$.

Assume $\ptrn P$ is not a $d$-pattern, this can only occur if some arc or curve $\trk t\in\ptrn P$ has a self-returning segment, i.e, a segment $\gamma\subseteq\trk t$ properly embedded in some 2-simplex $\simpf f\subset \simp K$ such that $\partial\gamma=\gamma \cap \simpe e$ for some edge $\simpe e\subset \simpf f$. Note that in this case $\simpe{e} \nsubseteq \partial M$ otherwise $\trk t$ is non-essential. Let $\gamma\subseteq \trk t\cap \simpf f$ be a self-returning segment with innermost endpoints amongst all curves and arcs in $\ptrn P$ (i.e, a self-returning segment whose endpoints do not bound two endpoints of another self-returning segment). Homotope $\trk t$ by pushing $\gamma$ to the 2-simplex on the other side of $\simpe e$. Note that since we chose an innermost $\gamma$ this homotopy does not create any new intersections in $\uc{\mfld S}$. Hence, the collection still satisfies all the assumption, and the total number of intersections of the curves and arcs of $\ptrn P$ with the one-skeleton of $\simp K$ decreased. Thus, after finitely many such moves the resulting collection is a $d$-pattern in $\mfld S$.

Now, by the main theorem we have

\begin{theorem}
\label{thm: finiteness of pattern on surfaces}
Let $\mfld M$ be a compact surface, with a triangulation $\simp K \homeo \mfld M$. There exists a constant $C$ such that if $\ptrn S$ is a collection of non-homotopic essential, 2-sided, properly immersed curves and arcs, such that the size of a pairwise intersecting collection of lifts to $\uc{\mfld M}$ is at most 2, then $|\ptrn S|\le C$.
\end{theorem}

\begin{proof}
By the above, one can replace the original collection of surface with homotopic curves which form a pattern on $\simp K$. By the main theorem, there exists $C$ such that the number of parallelism classes of tracks in $\ptrn P$ is bounded by $C$. Now it remains to note that if two curves (or arcs) are parallel then they are homotopic. And indeed, if they are parallel then (up to homotopy) we may assume that they bound an $I$-bundle region. But since they are 2-sided, it must imply that there exists a homotopy between them.
\end{proof}

We note that this theorem could also be proven by a simpler argument using the Euler characteristic.

\subsection{Patterns on 3-manifolds}

Recall the following definition from \cite{DiDu89}.

\begin{definition}
Let $\mfld M$ be an 3-manifold, and let $\simp K$ be a triangulation of $\mfld M$. An embedded surface $\smfld S\subset \mfld M$ is \emph{patterned} if $\smfld S \cap \simp K^{(2)}$ is a 1-pattern in $\simp K^{(2)}$ (the two skeleton of $\simp K$), and if for all 3-simplex $\simps \sigma\subset \simp K$, $\simps\sigma \cap \smfld S$ is a collection of disjoint embedded disks.
\end{definition}

Such a patterned surface $\smfld S$ is determined by $\smfld S \cap \simp K^{(2)}$  (see \cite{DiDu89}).

For our discussion let $\mfld M$ be a compact irreducible, boundary-irreducible 3-manifold. Let $\ptrn S$ be a collection of non-homotopic, incompressible, $\partial$-incom\-pres\-sible, 2-sided, properly embedded surfaces (in general position) in  $\mfld M$ , such that the size of a pairwise intersecting collection of their lifts to $\uc{\mfld M}$ is at most $d$.

We would like, as in the previous section, to homotope the surfaces such that each surface is patterned and $\ptrn S \cap \simp K^{(2)}$ is a $d$-pattern. Since the surfaces are embedded and satisfy that the size of a pairwise intersecting collection of their lifts to $\uc{\mfld M}$ is at most $d$, it is enough to prove that we can homotope them (while preserving the above properties) to surfaces such that the intersection of each one with $\simp K^{(2)}$ is a pattern. 

In the procedure defined in \cite{DiDu89}, the authors describe three types of moves, A, B and C,  transforming the embedding $f:\smfld S\to\mfld M$ to an embedding $f':\smfld S'\to\mfld M$, possibly changing the surface. 

Under our assumptions, we note that if $D$ is an embedded disk such that $D\cap\smfld S=\partial D$ then by incompressibility there exists a disk $D'\subset \smfld S$ such that $D'\cap D = \partial D'=\partial D$, thus, by irreducibility $D\cup D'$ bounds a 3-ball 
and one can homotope $\smfld S$ to $\smfld S'$ where $\smfld S' = (\smfld S \setminus D) \cup D'$. 

Similarly, if $D$ is an embedded disk whose boundary consists of two arcs $\alpha,\beta$ such that $\alpha = D\cap\smfld S, \beta=D\cap\partial \mfld M$ then by $\partial$-incompressibility there exist a disk $D'\subset \smfld S$  such that $D'\cap D = \alpha$ and 
$\partial (D\cup D') \subset \partial \mfld M$. Call $\tilde D$ the disk $D \cup D'$. Now by boundary irreducibility there exists a disk $D''\subset \partial \mfld M$ such that $D''\cap \tilde D = \partial D'' = \partial \tilde D$ and $D''\cup \tilde D''$ bounds a 3-ball and one can homotope $\smfld S$ to $\smfld S'$ where $\smfld S' = (\smfld S \setminus D) \cup D'$. 

Thus it follows that some of the moves defined in \cite{DiDu89} are not relevant in our case, and the every other can be made by a homotopy. In the figures \ref{surface},   \ref{typeA},   \ref{typeBi},   \ref{typeBii},   \ref{typeBiii} and   \ref{typeC}, we give a schematic description of some of the moves, adjusted to our case. For further details we refer the reader to \cite{DiDu89}.

By choosing innermost curves or arcs in the procedure, we guarantee that the surfaces obtained will remain embedded and will not have more intersections. This procedure terminates, see \cite{DiDu89}. Thus we obtain the following.

\begin{lemma}
\label{lem:collection to pattern}
Let $\mfld M$ and $\ptrn S$ be as above. Then one can choose representatives of $\ptrn S$ such each $\smfld S\in\ptrn S$ is patterned and $\ptrn S\cap K^{(2)}$ is a $d$-pattern.
\end{lemma}

Now, by a similar proof to that of Theorem~\ref{thm: finiteness of pattern on surfaces} using the main theorem and the previous lemma we get

\begin{theorem}
Let $\mfld M$ be a compact irreducible, boundary-irreducible 3-manifold. There exists a number $C$, depending only on $\mfld M$ such that if $\ptrn S$ is a collection of non-homotopic, $\pi_1$-injective, 2-sided, properly embedded surfaces (in general position) in  $\mfld M$ , such that the size of a pairwise intersecting collection of lifts to $\uc{\mfld M}$ is at most 2, then $|\ptrn S|\le C$.
\end{theorem}

\bibliographystyle{plain}
\bibliography{main}
\newpage

\begin{figure}[!ht]
\begin{center}
 \def\svgwidth{0.7\textwidth}
 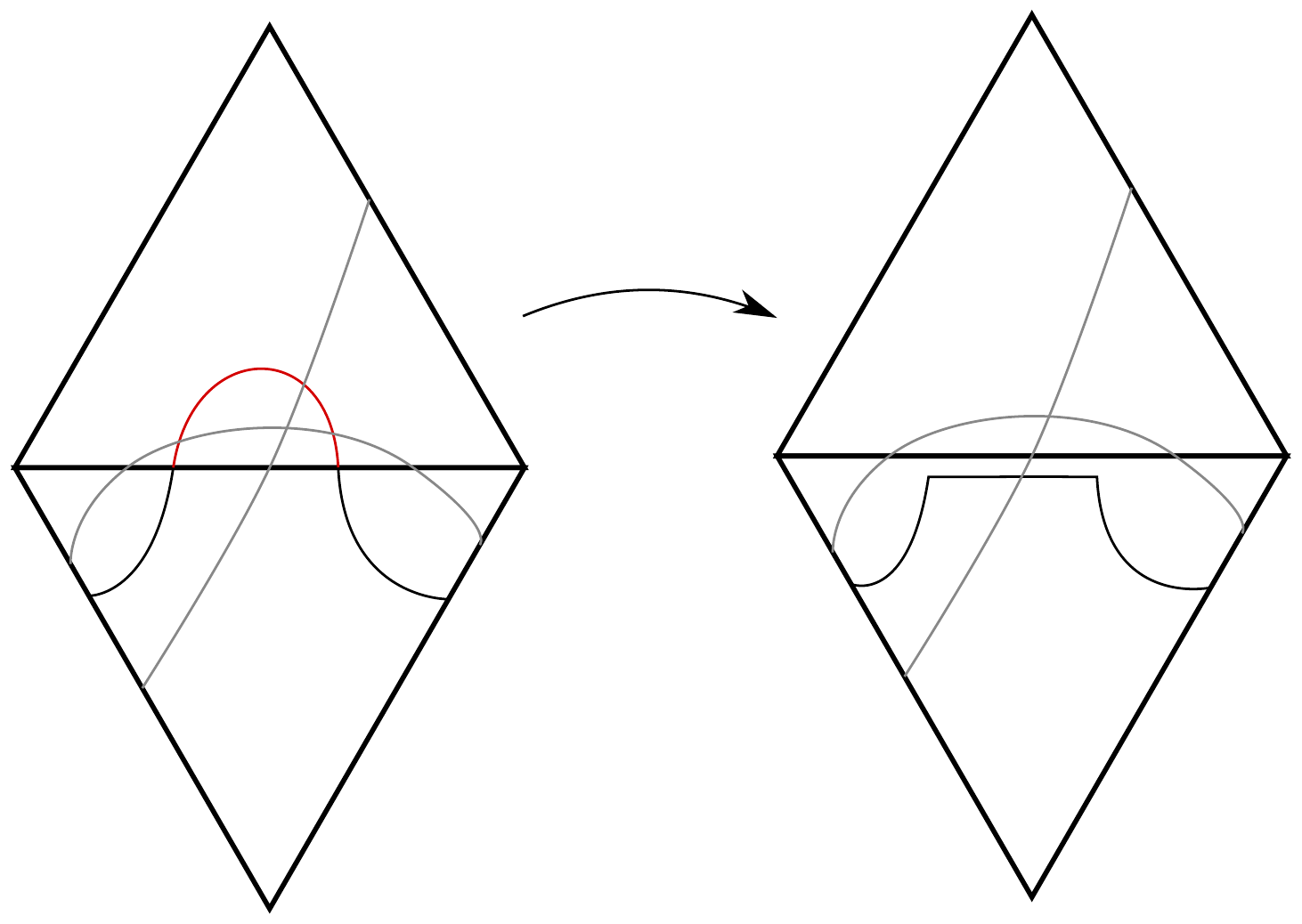
 \end{center}
 \caption{Fixing a self-returning curve in a surface}
 \label{surface}
\end{figure}

\begin{figure}[!ht]
\begin{center}
 \def\svgwidth{0.7\textwidth}
 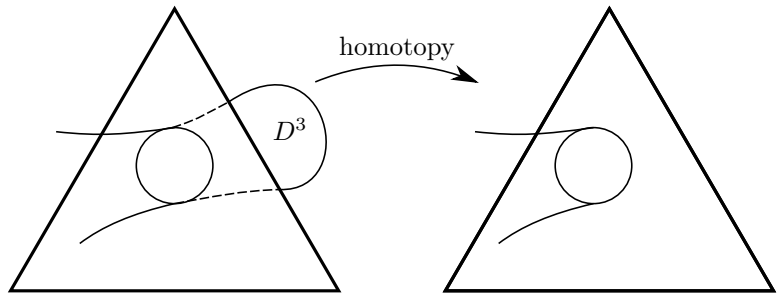
 \end{center}
 \caption{Fixing a closed simple curve in a face}
 \label{typeA}
\end{figure}

\begin{figure}[!ht]
\begin{center}
 \def\svgwidth{0.7\textwidth}
 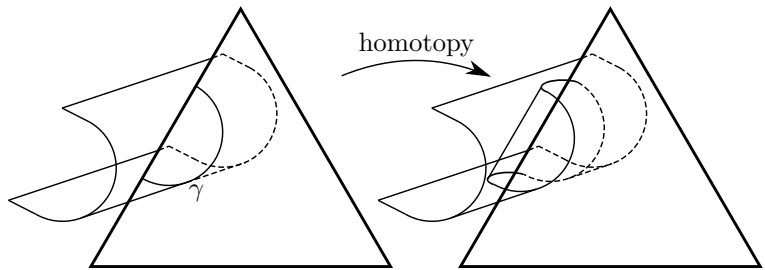
 \end{center}
 \caption{Fixing a self-returning curve $\gamma$, which returns to a non-boundary edge}
 \label{typeBi}
\end{figure}

\begin{figure}[!ht]
\begin{center}
 \def\svgwidth{0.7\textwidth}
 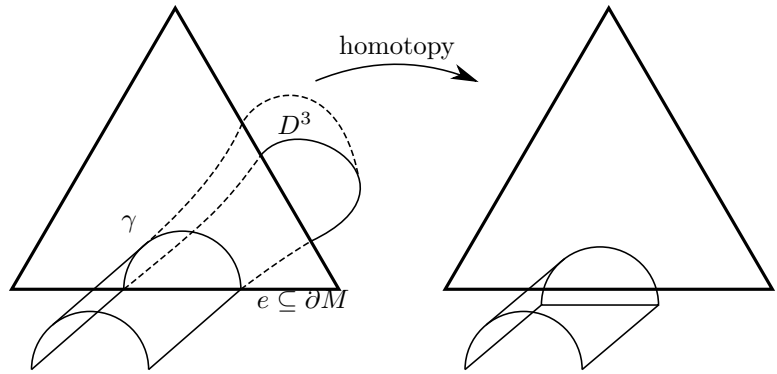
 \end{center}
 \caption{Fixing a self-returning curve $\gamma$, which is not in the boundary but returns to a boundary edge $\simpe{e}\subset \partial \mfld M$}
 \label{typeBii}
\end{figure}

\begin{figure}[!ht]
\begin{center}
 \def\svgwidth{0.7\textwidth}
 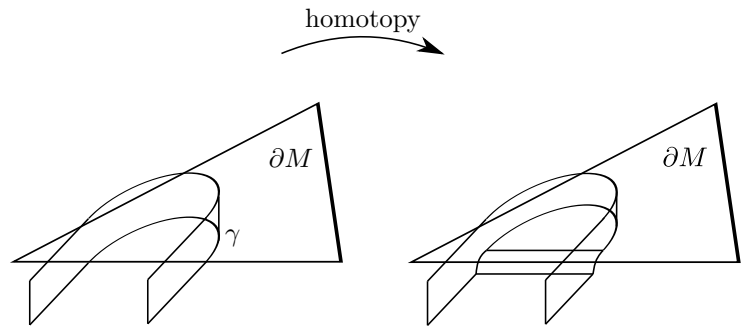
 \end{center}
 \caption{Fixing a self-returning curve $\gamma$, which is contained in the boundary of $\mfld M$}
 \label{typeBiii}
\end{figure}

\begin{figure}[!ht]
\begin{center}
 \def\svgwidth{0.7\textwidth}
 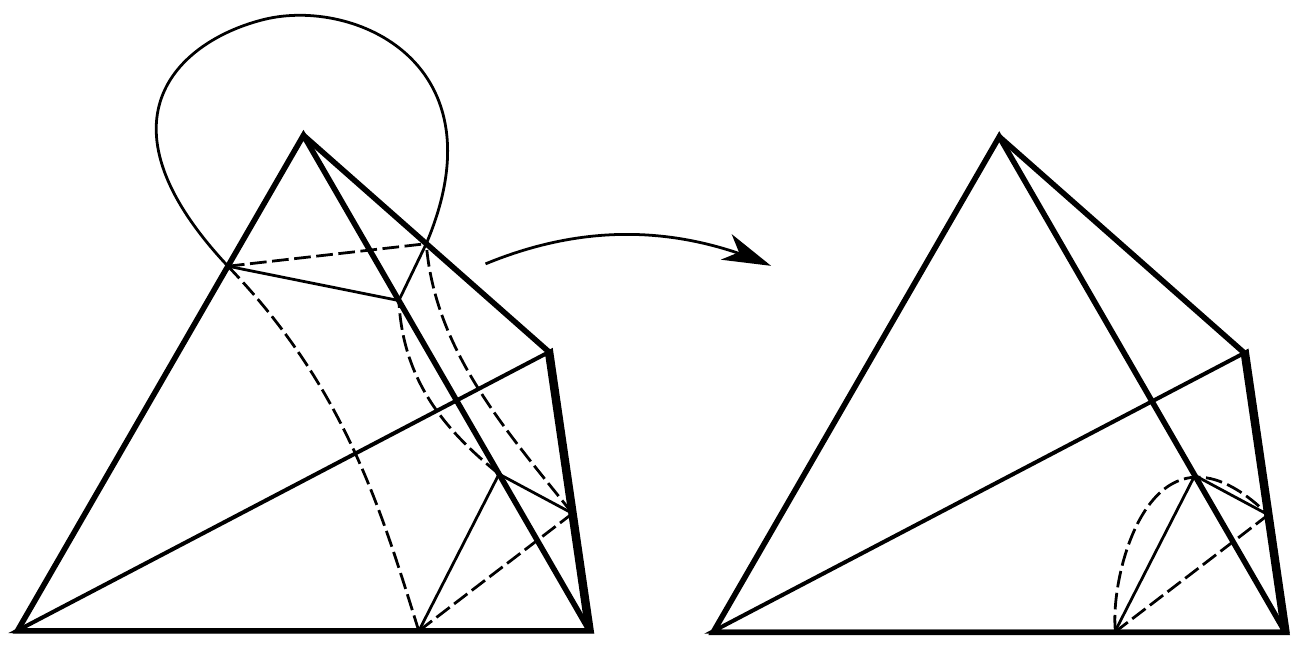
 \end{center}
 \caption{Fixing a non disc component of $\ptrn S \setminus \mfld M ^{(2)}$}
 \label{typeC}
\end{figure}

\end{document}